\documentclass[11pt,a4paper,reqno]{amsart}
\usepackage[a4paper,lmargin=3cm,rmargin=3cm,tmargin=4cm,bmargin=4cm]{geometry}

\usepackage[english]{babel}

\usepackage[centertags]{amsmath}
\usepackage{amsfonts}
\usepackage{amssymb}
\usepackage{amsthm}
\usepackage{dsfont}
\usepackage{mathtools}

\usepackage[pdftex]{color}
\usepackage[bookmarks=true,hyperindex,pdftex,colorlinks, citecolor=MidnightBlue,linkcolor=MidnightBlue,
urlcolor=MidnightBlue
]{hyperref}
\usepackage[dvipsnames]{xcolor}

\usepackage{enumerate}
\usepackage{enumitem}

\theoremstyle{plain}
\newtheorem{theorem}{Theorem}
\newtheorem{proposition}[theorem]{Proposition}

\theoremstyle{definition}
\newtheorem*{definition*}{Definition}

\theoremstyle{remark}

\newcommand{\K}{\mathbb{K}}
\newcommand{\R}{\mathbb{R}}

\newcommand{\C}{\mathbb{C}}

\DeclareMathOperator{\re}{Re}

\renewcommand{\geq}{\geqslant}
\renewcommand{\leq}{\leqslant}

\newcommand{\norm}[1]{\left\Vert#1\right\Vert}

\begin{document}

\title[On super Delta-points and the convex-DLD2P in absolute sums]{On super Delta-points and the convex-DLD2P in absolute sums}

\author[J. Guerrero-Viu]{Juan Guerrero-Viu}
\address[J. Guerrero-Viu]{Departamento de Matemáticas, Universidad de Zaragoza, 50009, Zaragoza, Spain} 
\email{j.guerrero@unizar.es}
\urladdr{\href{https://orcid.org/0009-0001-2125-5120}{ORCID: \texttt{0009-0001-2125-5120}}}

\author[J. Markowicz]{Joanna Markowicz}
\address[J. Markowicz]{Department of Mathematics, University of the National Education Commission, ul. Podchorażych 2, 30-084 Kraków, Poland}
\email{joanna.markowicz@uken.krakow.pl} \urladdr{\href{https://orcid.org/0000-0002-6688-5038}{ORCID: \texttt{0000-0002-6688-5038}}}

\subjclass[2020]{Primary 46B04; Secondary 46B20}

\keywords{Super Delta-points; convex diametral local diameter two property; absolute sum; Banach space}

\begin{abstract}
We partially answer two open questions concerning diameter two properties in absolute sums. First, we identify the conditions that a super $\Delta$-point in an absolute sum of Banach spaces imposes on the coordinates. Secondly, we show that the convex diametral local diameter two property (convex-DLD2P) passes from an absolute sum $X\oplus_N Y$ to its factors whenever $N$ is not the $\ell_\infty$-norm. 
\end{abstract}
\maketitle

Let $X$ be a Banach space over the scalar field $\K$, which can be either $\R$ or $\C$. The closed unit ball and the unit sphere of $X$ is denoted by $B_X$ and $S_X$, respectively. We write $X^*$ for the topological dual space of $X$. Moreover, given $x^*\in B_{X^*}$ and $\alpha>0$, the slice of $B_X$ produced by $x^*$ is the (open) set $S(B_X,x^*,\alpha)=\{ x\in B_X : \text{Re } x^*(x)>1-\alpha\}.$

Let $x\in S_X$. According to \cite{ahlp} and \cite{MPR}, it is said that 
\begin{enumerate}
    \item $x$ is a \textit{$\Delta$-point} if for every slice $S$ of $B_X$ with $x\in S$ and for every $\varepsilon>0$, there is some $y\in S$ such that $\norm{x-y}\geq 2-\varepsilon$;
    \item $x$ is a \textit{super $\Delta$-point} if for every relatively weakly open set $W$ of $B_X$ with $x\in W$ and for every $\varepsilon>0$, there is some $y\in W$ such that $\norm{x-y}\geq 2-\varepsilon$.
\end{enumerate}
These notions have received significant attention in recent years \cite{aalmmppv,ahlp,lrt,MPR,VeeorgStudia}, since they can be regarded as natural local versions of some of the classical diameter two properties. Standing in sharp opposition to the Radon–Nikodým property and the convex point of continuity property, diameter two properties ensure that certain subsets of the unit ball are as large as possible in terms of diameter. Recall that $X$ has the \textit{local diameter two property} (LD2P) if every slice of $B_X$ has diameter two; and the \textit{diameter two property} (D2P) whenever every non-empty relatively weakly open subset of $B_X$ has diameter two.  
More recently, several stronger variants of these notions were considered in \cite{blr18}. Namely, we say that
$X$ has the \textit{diametral local diameter two property} (\textit{DLD2P}) if given any slice $S$ of
$B_X$, $x\in S\cap S_X$ and $\varepsilon>0$, there
exists $y\in S$ such that
$\norm{x-y}\geq 2-\varepsilon$ (equivalently, if every $x\in S_X$ is a $\Delta$-point). On the other hand, it is said that $X$ has the \textit{diametral diameter two property}
(\textit{DD2P}) if given any relatively weakly open subset $W$ of
$B_X$, $x\in W\cap S_X$ and $\varepsilon>0$, there
exists $y\in W$ such that
$\norm{x-y}\geq2-\varepsilon$ (equivalently, if every $x\in S_X$ is a super $\Delta$-point).
There are also corresponding versions of the above properties involving convex combinations of slices of $B_X$, but we will not deal with them here.

Furthermore, lying between the LD2P and the DLD2P, and motivated by the problem of determining the minimum amount of $\Delta$-points required to ensure that every slice has diameter two, the following property was introduced in \cite{ahlp}. 
\begin{definition*}
Let $X$ be a Banach space. We say that $X$ has the \textit{convex-DLD2P} if $B_X$ is the closed convex hull of the set of $\Delta$-points, equivalently, if for every slice $S$ of $B_X$ there exists a $\Delta$-point $x\in S$.
\end{definition*}
All the aforementioned properties have been extensively studied and have generated a substantial amount of research over the last decades. We refer the reader to \cite{ahntt,blr15eje1,blr18,MPR} and their references for background and further information on this topic. 

However, some natural questions remain open. In this paper, we address two problems, both concerning absolute sums. Recall that an \emph{absolute normalised norm} $N$ on $\R^2$ is a norm satisfying the following conditions:
\begin{enumerate}
    \item $N(a,b)=N(|a|,|b|)$ for all $(a,b)\in \R^2$;
    \item $N(1,0)=N(0,1)=1$.
\end{enumerate}
Given Banach spaces $X,Y$ and an absolute normalized norm $N$ on $\mathbb{R}^2$, the \emph{absolute sum} of $X$ and $Y$ with respect to $N$, denoted by \(X\oplus_N Y\), is the product space $X\times Y$ endowed with the norm
\[
\norm{(x,y)}_N=N(\|x\|,\|y\|), \qquad (x,y)\in X\times Y.
\]
Important examples are the $\ell_p$-sums ($1\leq p\leq \infty$), corresponding to the case where $N$ is the usual $\ell_p$-norm on $\mathbb{R}^2$; in this case, we write $X\oplus_p Y$. Furthermore, we have $(X\oplus_N Y)^*=X^*\oplus_{N^*} Y^*$, where $N^*$ denotes the dual norm of $N$, which is given by
$$N^*(c,d)=\max_{N(a,b)\leq 1} (|ac|+|bd|), \qquad  (c,d)\in \R^2.$$

Our first result provides an answer to a question raised in \cite{MPR} concerning super $\Delta$-points. More precisely, given $x\in S_X$, $y\in S_Y$, and $a,b\geq 0$ with $N(a,b)=1$, the authors asked which conditions on $x$ and $y$ are forced by the fact that $(ax,by)$ satisfies one of the six standard diametral notions \cite[Question 7.10]{MPR}. While the cases of Daugavet and $\Delta$-points were already known (see \cite[tables on pp. 86-87]{Pirkthesis}), we make further progress on this problem by addressing the case of super $\Delta$-points.

\begin{theorem}\label{thm:superDeltatosummands}
    Let $X,Y$ be Banach spaces and let $N$ be an absolute normalised norm on $\R^2$. Given $x\in S_X$, $y\in S_Y$ and $a,b\geq 0$ with $N(a,b)=1$, suppose that $(ax,by)$ is a super $\Delta$-point in $X\oplus_N Y$.
    \begin{itemize}
        \item[\upshape(a)] If $b\neq1$, then $x$ is a super $\Delta$-point.
        \item[\upshape(b)] If $a\neq 1$, then $y$ is a super $\Delta$-point.
        \item[\upshape(c)] If $a=b=1$, then $x$ or $y$ is a super $\Delta$-point.
    \end{itemize}
\end{theorem}

\begin{proof}
    (a). We adapt the proof of \cite[Theorem 3.4.4]{Pirkthesis}. If $b\neq 1$, then $a\neq 0$. Let $c,d\geq 0$ such that $N^*(c,d)=1$ and $ac+bd=1$. Suppose that $x$ is not a super $\Delta$-point. Then, there exist $\varepsilon>0$ and a weak open set $W$ in $B_X$ with $x\in W$, such that $\norm{\widetilde{x}-x}< 2-\varepsilon$, for all $\widetilde{x}\in W$. 
    Pick any $x_1^*,\ldots,x_n^*\in S_{X^*}$ and $\alpha_1,\ldots,\alpha_n\in (0,1)$ such that $$x\in \bigcap_{i=1}^n S(B_X,x_i^*,\alpha_i)\subseteq W.$$ Thanks to \cite[Lemma 2.1]{IK} we may consider $\alpha\coloneq \alpha_1=\cdots = \alpha_n$. Pick $y^*\in S_{Y^*}$ with $y^*(y)=1$ and consider for each $i\in \{1,\ldots,n\}$, the functional $f_i\coloneqq (cx_i^*, (1-\alpha)d y^*)\in X^*\oplus_{N^*}Y^*$ which satisfies 
    $$
    \|f_i\|=N^*(c,(1-\alpha)d)\leq N^*(c,d)=1.
    $$
    Since $$\re f_i(ax,by)= ac(\re x_i^*(x))+(1-\alpha)bd(\re y^*(y))>(1-\alpha)(ac+bd)=1-\alpha,$$
    we may choose $\beta,\gamma >0$ such that $\beta<\min\{a\varepsilon, \gamma \varepsilon\}$ and $\text{Re }f_i(ax,by)>1-(\alpha-\gamma)$, for all $i\in \{1,\ldots,n\}$. Furthermore, using \cite[Lemma 1.4.14]{Pirkthesis}, pick $\delta>0$ such that, for every $p,q,r\geq 0$, if $$2-\delta\leq N(p,q)\leq N(r,q)\leq 2\quad \text{and}\quad q<2-\delta,$$ then $|p-r|<\beta$. There is no loss of generality in assuming that $b<1-\delta$. Consider now the relatively weakly open set $$\widetilde{W}\coloneqq \bigcap_{i=1}^n S(B_{X\oplus_N Y},f_i,\alpha-\gamma),$$
    which clearly contains $(ax,by)$. Since it is a super $\Delta$-point, we can find $(\widetilde{x},\widetilde{y})\in \widetilde{W}$ with $$\norm{(ax,by)-(\widetilde{x},\widetilde{y})}_N\geq 2-\delta.$$ Hence, for each $i\in \{1,\ldots,n\}$,
    \begin{align*}
        c(\re x_i^*(\widetilde{x}))+(1-\alpha)d\norm{\widetilde{y}}&\geq c(\re  x_i^*(\widetilde{x}))+(1-\alpha)d (\re  y^*(\widetilde{y}))
        \\&= \re f_i(\widetilde{x},\widetilde{y}) >1-(\alpha-\gamma)>1-\alpha\\&\geq (1-\alpha)(c\norm{\widetilde{x}}+d\norm{\widetilde{y}}),
    \end{align*}
    which proves $$c(\re x_i^*(\widetilde{x}))>(1-\alpha)c\norm{\widetilde{x}}.$$ In other words,
    $$\frac{\widetilde{x}}{\norm{\widetilde{x}}}\in \bigcap_{i=1}^n S(B_X,x_i^*,\alpha_i)\subseteq W.$$ Therefore, by hypothesis, we know that $\norm{\frac{\widetilde{x}}{\norm{\widetilde{x}}}-x}<2-\varepsilon$. Let us show that $\norm{ax-\widetilde{x}}<a+\norm{\widetilde{x}}-\beta$. We consider two cases. If $\norm{\widetilde{x}}\geq a$, we have
    \begin{align*}
        \norm{ax-\widetilde{x}}&\leq \norm{ax-a\frac{\widetilde{x}}{\norm{\widetilde{x}}}}+\norm{a\frac{\widetilde{x}}{\norm{\widetilde{x}}}-\widetilde{x}}
        \\&< a(2-\varepsilon)+|a-\norm{\widetilde{x}}|
        \\&=a+\norm{\widetilde{x}}-a\varepsilon
        <a+\norm{\widetilde{x}}-\beta.
    \end{align*}
    On the other hand, if $a\geq \norm{\widetilde{x}}$, then
    \begin{align*}
        c\norm{\widetilde{x}}+(1-\alpha)d\norm{\widetilde{y}}&\geq c(\re  x_i^*(\widetilde{x}))+(1-\alpha)d(\re  y^*(\widetilde{y}))\\&= \re f_i(\widetilde{x},\widetilde{y})>1-\alpha+\gamma \\&\geq (1-\alpha)d\norm{\widetilde{y}}+\gamma,
    \end{align*}
    which yields to $\norm{\widetilde{x}}\geq c\norm{\widetilde{x}}>\gamma$. Hence,
    \begin{align*}
        \norm{ax-\widetilde{x}}&\leq \norm{ax-\norm{\widetilde{x}}x}+\norm{\norm{\widetilde{x}}x-\widetilde{x}}\\&\leq a-\norm{\widetilde{x}}+\norm{\widetilde{x}}(2-\varepsilon)
        \\&=a+\norm{\widetilde{x}}-\norm{\widetilde{x}}\varepsilon< a+\norm{\widetilde{x}}-\beta.
    \end{align*}
    Finally, it follows that
    \begin{align*}
        N(a+\norm{\widetilde{x}}-\beta,b+\norm{\widetilde{y}})&\geq N(\norm{ax-\widetilde{x}},\norm{by-\widetilde{y}})=\norm{(ax,by)-(\widetilde{x},\widetilde{y})}_N\geq 2-\delta,
    \end{align*}
    so
    $$2-\delta\leq  N(a+\norm{\widetilde{x}}-\beta,b+\norm{\widetilde{y}})\leq N(a+\norm{\widetilde{x}},b+\norm{\widetilde{y}})\leq 2.$$
    However, recall that $b+\norm{\widetilde{y}}<2-\delta$, so by the choice of $\delta$ we obtain $$\beta=|(a+\norm{\widetilde{x}}-\beta)-(a+\norm{\widetilde{x}})|<\beta,$$
    which is clearly a contradiction.

    (b). It is analogous to a).

    (c). If $a=b=1$, then $N$ is the $\ell_\infty$-norm. Hence, suppose that both $x$ are $y$ are not super $\Delta$-points. Then, there is $\varepsilon>0$ and weak open sets $W,V$ in $B_X$ and $B_Y$ respectively such that $x\in W$, $y\in V$ and $$\norm{(x,y)-(\widetilde{x},\widetilde{y})}_\infty<2-\varepsilon,\quad  \forall (\widetilde{x},\widetilde{y})\in W\times V.$$
    Since $W\times V$ is a weak open set in $B_{X\oplus_\infty Y}$, we conclude that $(x,y)$ is not super $\Delta$.
\end{proof}

The second question we address concerns the convex-DLD2P. In \cite[Problem~1]{ahlp} (see also \cite[Problem~9]{Pirkthesis}), it was asked whether the convex-DLD2P in $X\oplus_N Y$ forces the factors to have the same property. We answer this question for every absolute sum different from the $\ell_\infty$-sum.

The main difficulty in dealing with the $\ell_\infty$-sum stems from the existence of $\Delta$-points whose coordinates need not be $\Delta$-points in the corresponding factors (see \cite[Example~3.4.6]{Pirkthesis}).

\begin{proposition}\label{prop:convextosummands}
    Let $X$ and $Y$ be Banach spaces and let $N$ be an absolute normalised norm on $\R^2$ different from the $\ell_\infty$-norm.  If $X\oplus_N Y$ has the convex-DLD2P, then both $X$ and $Y$ have the convex-DLD2P.
\end{proposition}

\begin{proof}
Suppose that $X$ fails the convex-DLD2P (a similar argument is valid for $Y$), so there is some non-empty slice $S=S(B_X,f,\alpha)$ that does not contain any $\Delta$-point. Since $N$ is different from the $\ell_\infty$-norm, consider $t:=\max\{ s\in [0,1]\colon  N(s,1)\leq1\}$ which clearly verifies $t<1$. Hence, up to consider a smaller slice inside $S$ if necessary, we may assume that $1-\alpha>t$. Next, define the slice
        $$\widetilde{S}:=\left\{(x,y)\in B_{X\oplus_N Y} \colon x\in S \right\}=S(B_{X\oplus_N Y}, (f,0), \alpha),$$
        which is non-empty because $S\neq \emptyset$. Therefore, there is some $\Delta$-point $(ax,by)\in \widetilde{S}$, where $x\in S_X$, $y\in S_Y$, $a,b\geq 0$ and $N(a,b)=1$. Furthermore, as $ax\in S$, then $$a=\norm{ax}\geq \re  f(ax)>1-\alpha>t.$$
        Thus, by the maximality of $t$, we have $b\neq 1$ which yields to $x$ being a $\Delta$-point (see \cite[Theorem 3.4.4]{Pirkthesis}). Finally, $$\re  f(x)=\frac{1}{a} \re  (ax)>\frac{1}{a}(1-\alpha)>1-\alpha,$$
        so $x\in S$, leading to a contradiction.   
\end{proof}

\section*{Acknowledgements}  

The authors are grateful to Miguel Martín and Abraham Rueda Zoca for their comments.

Part of this work was conducted during the first-named author’s visit to Universidad de Granada in March 2026 supported by ``Maria de Maeztu'' Excellence Unit IMAG (CEX2020-001105-M), funded by MICIU/AEI/10.13039/501100011033  and by Junta de Andaluc\'ia. The author is grateful for the hospitality and support received.

The research of J.\ Guerrero-Viu was supported by FPU24/02284 predoctoral grant funded by MCIU and by grant PID2022-137294NB-I00 funded by  MCIN\slash AEI\slash10.13039\slash501100011033 and by “ERDF A way of making Europe”. 
The research of J. Markowicz was partially supported by The Excellent Mobility DNa.711/IDUB/EM/2024/01/00023 at the University of the National Education Commission. 

The authors acknowledge the use of artificial intelligence–based tools, including ChatGPT (OpenAI), and Gemini (Google), for language editing and stylistic improvements of the manuscript. The authors remain fully responsible for the content of this work.

\end{document}